\documentclass[11pt,reqno]{amsart}
\usepackage{amsmath}  
\usepackage{amsfonts}  
\usepackage{amssymb}  
\usepackage{amsthm}  
\usepackage{multicol}
\newcommand{\R}{\mathbb R}  

\renewcommand{\b}{\big}

\newcommand{\Z}{\mathbb Z}  
\renewcommand{\H}{\mathbb H}  

\newcommand{\hf}{\frac{1}{2}}  
\newcommand{\intR}{\int_{-\infty}^{\infty}}
\renewcommand{\l}{\left}
\renewcommand{\r}{\right}
\newcommand{\generic}{\left( \begin{smallmatrix} a & b \\ c & d
 \end{smallmatrix} \right)}
\newcommand{\eps}{\varepsilon}
\renewcommand{\Im}{{\mathrm{Im\,}}}  

\newcommand{\sm}[4]{\l(\begin{smallmatrix} #1 & #2 \\ #3 & #4\end{smallmatrix}\r)}

\newcommand{\mm}{\,{\mathrm{mod}}\,}
\newcommand{\E}{{\mathcal E}}

\newcommand{\Rr}{{\mathcal R}}

\newcommand{\Lan}{\,\big\langle}
\newcommand{\Ran}{\big\rangle\,}

\renewcommand{\b}{\big}

\renewcommand{\include}{\input}

\begin{document}

\theoremstyle{plain}
\newtheorem{theorem}{Theorem}[section]
\newtheorem{proposition}[theorem]{Proposition}
\newtheorem{prop}[theorem]{Proposition}
\newtheorem{lemma}[theorem]{Lemma}
\newtheorem{lem}[theorem]{Lemma}
\newtheorem{fact}[theorem]{Fact}
\newtheorem{corollary}[theorem]{Corollary}
\newtheorem{remark}[theorem]{Remark}
\theoremstyle{definition}
\newtheorem{definition}[theorem]{Definition}
\numberwithin{equation}{section}
\newcommand{\ch}{{\mathrm{char}}}
\newcommand{\Wig}{{\mathrm{Wig}}}
\newcommand{\sqf}{\,{\mathrm{squarefree}}\,}
\newcommand{\Sqf}{{\mathrm{Sqf}}}
\newcommand{\Qs}{Q^{\hf}}
\newcommand{\Qms}{Q^{-\hf}}

\newcommand{\Zert}{\b\bracevert}

\newcommand{\Fymp}{{\mathcal F}^{\mathrm{symp}}}
\newcommand{\Feuc}{{\mathcal F}^{\mathrm{euc}}}
\newcommand{\Conj}{{\mathcal C}^{\mathrm{conj}}}
\newcommand{\K}{{\mathfrak K}}

\newcommand{\slut}{\end{document}}

\title[A unified scheme of approach to Ramanujan conjectures]{A unified scheme of approach to Ramanujan conjectures}

\author{Andr\'e Unterberger, University of Reims, CNRS UMR9008}

Math\'ematiques, Universit\'e de Reims, BP 1039, F51687 Reims Cedex, France, andre.unterberger@univ-reims.fr\\

\maketitle

{\sc Abstract.} The Ramanujan conjecture for modular forms of holomorphic type was proved by Deligne \cite{del} almost half a century ago: the proof, based on his proof of Weil's conjectures, was an achievement of algebraic geometry. We give here a short proof of the Ramanujan-Deligne theorem, and we shall indicate at the end the identity of the proof with that of the Ramanujan-Petersson conjecture for Maass forms \cite{untRam}, save for the necessity of some spectral-theoretic developments in the latter case.\\

\section{Introduction}\label{sec1}

This is an analyst's proof of the classical Ramanujan conjecture for modular forms
of the group $\Gamma=SL(2,\Z)$ of holomorphic type. Deligne's proof, valid for congruence groups, was obtained as a consequence of his proof of the Weil conjectures. We only consider the Ramanujan conjecture here. The short proof to follow is a belated answer to a question raised by Chandrasekharan \cite[p.140]{cha} or implicitly raised by Manin \cite[p.99]{man}.\\

We limit ourselves, here, to modular forms of even weight $m+1=2,4,\dots$ for the full unimodular group. These are holomorphic functions $f$ in the hyperbolic half-plane $\H$ satisfying the identity $f(z)=(bz +d)^{-m-1}f\l(\frac{az+c}{bz+d}\r)$ for $\generic \in \Gamma=SL(2,\Z)$. Among these, cusp-forms are distinguished by the fact that they admit Fourier expansions of the kind $f(z)=\sum_{n\geq 1} b_n\,e^{2i\pi nz}$: then, $m+1\geq 12$. Finally, Hecke eigenforms are cusp-forms which are joint eigenfunctions of the Hecke operators $T_p$ ($p$ prime), the definition of which will be recalled. The Ramanujan-Deligne theorem is the inequality $\l|\frac{b_p}{b_1}\r|\leq 2\,p^{\frac{m}{2}}$ for $p$ prime. The function $f$ is normalized in Hecke's way if $b_1=1$, so that one has simply $T_pf=b_pf$ \cite[p.101]{iwatop}.\\

The main point of the proof consists in realizing automorphic functions as distributions in the plane, as explained in the next section: they are then regarded as the integral kernels of operators of interest. As will be very briefly indicated in a last section, the same method works in the non-holomorphic case \cite{untRam}, but it is not the same plane in the two cases: it is the one provided with its Euclidean structure in the present case, the one provided with the symplectic structure in the non-holomorphic case. It is not the same operator-calculus either: the Weyl symbolic (pseudodifferential) calculus has to be used in the non-holomorphic case.\\


\section{A representation of $SL(2,\R)$}\label{sec2}

In this section, we define the representation $\Omega$ and the operators $\Theta_m$ which isolate the individual terms of its decomposition into irreducibles. Also, we introduce for $M=1,2,\dots$ an object ${\mathfrak T}_M$, not quite an element of ${\mathcal S}'(\R^2)$ but just as good for our purposes (\ref{219}), the image of which under $\Theta_m$ is the Poincar\'e series
\begin{equation}\label{21}
\l(\Theta_m{\mathfrak T}_M\r)(z)=-i\,(2M)^{\frac{m}{2}}\sum_{(a,c)=1}
(-dz+b)^{-m-1}\exp\l(2i\pi M\frac{-cz+a}{-dz+b}\r).
\end{equation}
The integer $M$ is kept fixed until the very last section, and no uniformity with respect to $M$ is needed, or claimed, in all estimates.\\

We denote as $\Omega$ the representation of $SL(2,\R)$ in any of the spaces ${\mathcal S}(\R^2),\,L^2(\R^2),\,{\mathcal S}'(\R^2)$, unitary in the second case, defined on generators by the equations, in which the Euclidean Fourier transform is defined as $(\Feuc h)(x)=\int_{\R^2}h(y)\,e^{-2i\pi(x,y)}dx$,
\begin{align}\label{22}
{\mathrm (i)\ } \,&\Omega\l(\l( \begin{smallmatrix}
1 & 0 \\ c & 1\end{smallmatrix}\r)\r)\,h)(x)=h(x)\,e^{i\pi c |x|^2}, \quad x\in \R^2\,;\nonumber\\
{\mathrm (ii)\ } \,&\Omega\l(\l(\begin{smallmatrix}0 & 1 \\ -1 &
0\end{smallmatrix}\r)\r)\,h=-i\,{\Feuc}h\,;\nonumber\\
{\mathrm (iii)\ } \,&(\Omega\l(\l( \begin{smallmatrix}
a & 0 \\ 0 & a^{-1}\end{smallmatrix}\r)\r)\,h)(x)=a^{-1}\,h(a^{-1}x),
\quad x\in \R^2,\ a\neq 0\,.
\end{align}\\

We first show that the representation $\Omega$ contains all representations ${\mathcal D}_{m+1}$ from the holomorphic discrete series of representations of $SL(2,\R)$.\\

\begin{proposition}\label{prop21}
Given $m=1,2,\dots$ and $h\in {\mathcal S}(\R^2)$, set for $z$ in the hyperbolic half-plane $\H=\{z\colon \Im z>0\}$
\begin{equation}\label{23}
(\Theta_m\,h)(z)=\int_{\R^2} (x_1+ i\,x_2)^m \, e^{i\pi z\,|x|^2}\,
h(x)\,dx.
\end{equation}
For every $g=\generic \in SL(2,\R)$, one has $\Theta_m\l(\Omega(g)\,h\r)={\mathcal D}_{m+1}(g)\,\Theta_mh$, with
\begin{equation}\label{24}
({\mathcal D}_{m+1}\l(\generic\r)\,f)(z)=(bz+d)^{-m-1}f\l(\frac{az+c}{bz+d}\r)\,.
\end{equation}\\
\end{proposition}

\begin{proof}
One may assume that $g$ is one of the generators of $SL(2,\R)$ listed in (\ref{22}): the only non-immediate case is (ii). Writing
\begin{align}
\l(\Theta_m\l(-i\,\Feuc h\r)\r)(z)&=
-i\int_{\R^2}(x_1+ix_2)^me^{i\pi z\,|x|^2}\l(\Feuc h\r)(x)\,dx\nonumber\\
&=-i\Lan \Feuc\l((x_1+ix_2)^me^{i\pi z\,|x|^2}\r),\,h\Ran\,,
\end{align}
using the classical formula for the Fourier transform of products of radial functions by ``spherical harmonics'' (e.g.\,\cite{ste})
\begin{equation}\label{25}
\Feuc\l((x_1+ix_2)^mf(|x|)\r)=(x_1+ix_2)^m\,\times\,2\pi\,i^{-m}|x|^{-m}\int_0^{\infty}f(t)\,t^{m+1}
J_m(2\pi t\,|x|)\,dt
\end{equation}
and the equation \cite[p.93]{mos}
\begin{equation}
\int_0^{\infty} e^{i\pi zt^2}t^{m+1}J_m(2\pi\,|x|\,t)\,dt=
\frac{1}{2\pi}\,|x|^m(-iz)^{-m-1}\exp\l(-\frac{i\pi\,|x|^2}{z}\r).
\end{equation}
one obtains
\begin{equation}
\l(\Theta_m\l(-i\,\Feuc h\r)\r)(z)=z^{-m-1}\l(\Theta_mh\r)\l(-\frac{1}{z}\r),
\end{equation}
which is the desired case of (\ref{24}).\\
\end{proof}

Analysis in the plane $\R^2$ will be based on the use of the commuting operators (where
$\partial_j=\frac{\partial}{\partial x_j}$)
\begin{equation}
2i\pi\Rr=x_2\,\partial_1-x_1\,\partial_2,\qquad 2i\pi\E=1+x_1\,\partial_1 +x_2\,\partial_2\,,
\end{equation}
the first of which commutes with all transformations $\Omega(g)$. The second one will be used mostly in the form $(t^{2i\pi\E}\,h)(x_1,x_2)=t\,h(tx_1,tx_2)$ for $t>0$.\\

The following elementary, if admittedly tedious, calculation introduces the distribution $\psi_M$ the transforms of which under the operators $\Omega (g)$ with $g\in \Gamma$ constitute the individual terms of the Poincar\'e-type series ${\mathfrak T}_M$ alluded to in the introduction, to be defined in (\ref{219}) below.\\

\begin{lemma}
With $M=1,2,\dots$, let $\phi_M(x_1,x_2)=e^{2i\pi x_1\sqrt{2M}}$ and $\psi_M=\Omega\l(\sm{0}{1}{-1}{0}\r)\phi_M$. One has
\begin{equation}\label{29}
\psi_M(x)=-i\,\delta(x_1-\sqrt{2M})\delta(x_2).
\end{equation}
If $\sm{a}{\centerdot}{c}{\centerdot} \in SL(2,\Z)$ and $ac\neq 0$, one has
\begin{align}\label{211}
\bigg(\Omega\l(\sm{\centerdot}{-a}{\centerdot}{-c}\r)\,\psi_M\bigg)(x)
&=a^{-1}\,\exp\l(\frac{2i\pi M\overline{c}}{a}\r)\nonumber\\
&\times\,\exp\l(\frac{2i\pi x_1\sqrt{2M}}{a}\r)\,\exp\l(\frac{i\pi c\,|x|^2}{a}\r)\,,
\end{align}
where $\overline{c}$ is defined by the congruence $c\,\overline{c}\equiv 1\mm a$. Also, with $\eps=\pm 1$,
\begin{equation}\label{212}
\bigg( \Omega\l(\sm{\centerdot}{0}{\centerdot}{\eps}\r)\,\psi_M\bigg)(x)=
\eps\,\psi_M(\eps x),\qquad \l(\Omega\l(\sm{\centerdot}{-1}{1}{0}\r)h\r)(x)=-i\,(\Feuc h)(-x).
\end{equation}
while $\Omega\l(\sm{\centerdot}{1}{-1}{0}\r)$ is given in ({\em (\ref{23})\/}).\\
\end{lemma}

\begin{proof}
To obtain (\ref{29}), one writes
\begin{equation}
\psi_M(x)=-i\l(\Feuc \phi_M\r)(x)=-i\,\delta(x_1-\sqrt{2M})\delta(x_2).
\end{equation}
Some preparation will help the proof of (\ref{211}). One notes that
\begin{multline}
\Omega\l(\sm{1}{0}{\gamma}{1}\r)\psi_M=e^{2i\pi M\gamma}\psi_M,\\
\l(\Omega\l(\sm{1}{\beta}{0}{1}\r)\psi_M\r)(x)=-\frac{1}{\beta}\,\exp\l(\frac{2i\pi M}{\beta}\r) \exp\l(-2i\pi\frac{x_1\sqrt{2M}}{\beta}\r) \exp\l(\frac{i\pi\,|x|^2}{\beta}\r).
\end{multline}
The first equation is a consequence of (\ref{22}), and the second is obtained from the decomposition
\begin{equation}
\sm{1}{\beta}{0}{1}=\sm{0}{1}{-1}{0}\sm{1}{0}{-\beta}{1} \sm{0}{-1}{1}{0}
\end{equation}
from which, applying (\ref{22}) again,
\begin{align}
\l(\Omega\l(\sm{1}{\beta}{0}{1}\r) \psi_M\r)(x)&=
-i\,\Feuc\big[e^{-i\pi \beta\,|x|^2}e^{2i\pi x_1\sqrt{2M}}\big]\nonumber\\
&=-i\,\Feuc\l(y\mapsto e^{-i\pi \beta\,|y|^2}\r)(x_1-\sqrt{2M},\,x_2).
\end{align}\\

Then, using if $ac\neq 0$
\begin{equation}
\sm{b}{-a}{d}{-c}=\sm{-c^{-1}}{0}{0}{-c} \sm{1}{ac}{0}{1} \sm{1}{0}{-\frac{d}{c}}{1},
\end{equation}
one has
\begin{equation}\label{218}
\Omega\l(\sm{b}{-a}{d}{-c}\r)\psi_M=\exp\l(-\frac{2i\pi Md}{c}\r)\,\Omega\l(\sm{-c^{-1}}{0}{0}{-c} \sm{1}{ac}{0}{1}\r) \psi_M.
\end{equation}
The equations which precede give (\ref{212}).\\

The special cases when $ac=0$ are easily dealt with.\\
\end{proof}

Setting $N=\{\sm{1}{n}{0}{1}\colon n\in \Z\}$ and
$N^{\bullet}=\{\sm{1}{0}{n}{1}\colon n\in \Z\}$, one can define formally the series
\begin{multline}\label{219}
{\mathfrak T}_M=\sum_{g\in \Gamma/N}\Omega(g)\,\phi_M=
\sum_{g\in \Gamma/N^{\bullet}} \Omega(g)\,\psi_M\\
=\sum_{(a,c)=1} \Omega\l(\sm{\centerdot}{-a}{\centerdot}{-c}\r) \psi_M=\sum_{(a,c)=1}I_{a,c},
\end{multline}
with
\begin{equation}\label{220}
I_{a,c}(x)=a^{-1}\exp\l(\frac{2i\pi M\overline{c}}{a}\r)\,
\exp\l(\frac{2i\pi x_1\sqrt{2M}}{a}\r)\,\exp\l(\frac{i\pi c\,|x|^2}{a}\r).
\end{equation}
However, the series does not converge weakly in ${\mathcal S}'(\R^2)$: to recover this, it will be necessary to consider in place of this would-be distribution its image under the operator $(2i\pi\Rr)^4$. Also, the core of this proof of the Ramanujan-Deligne theorem will consist in $q$-dependent estimates (the cases when $q\to \infty$ and $q\to 0$ must both be considered) for the series
\begin{equation}\label{221}
\sum_{(a,c)=1}\Lan I_{a,c},\,q^{1-2i\pi\E}h\Ran
=\sum_{(a,c)=1}\Lan q^{1+2i\pi\E}\,I_{a,c},\,h\Ran.
\end{equation}
and its images under Hecke operators. Recall that $(q^{1-2i\pi\E}h)(x_1,x_2)=h\l(\frac{x_1}{q},\frac{x_2}{q}\r)$.\\


As the coefficient $\exp\l(\frac{2i\pi M\overline{c}}{a}\r)$ is unitary and would not change estimates, we shall consider instead the function
\begin{equation}\label{222}
\overset{\circ}{I}_{a,c}(x)=a^{-1}\,
\exp\l(\frac{2i\pi x_1\sqrt{2M}}{a}\r)\,\exp\l(\frac{i\pi c\,|x|^2}{a}\r).
\end{equation}\\

\section{Integral kernels and operators}

Given ${\mathfrak S}\in {\mathcal S}'(\R^2)$, we denote as $\K({\mathfrak S})$ the operator with integral kernel ${\mathfrak S}$, i.e., the one defined by the equivalent equations
\begin{align}
(\K({\mathfrak S})\,u)(x_1)&=\intR {\mathfrak S}(x_1,\,x_2)\,u(x_2)\,dx_2,\nonumber\\
<v,\,\,\K({\mathfrak S})\,u>&=\int_{\R^2} v(x_1)\,\,{\mathfrak S}(x_1,x_2)\,\,u(x_2)\,dx_1\,dx_2=<{\mathfrak S},\,v\otimes u>.
\end{align}
Note that we use here the bilinear version $<v,\,u>=\intR v(t)\,u(t)\,dt$, not the hermitian scalar product the consideration of which is appropriate in the non-holomorphic case.\\

Letting $P=\frac{1}{2i\pi}\,\frac{d}{dt}$, and $Q$ denoting the operator that multiplies functions of $t$ by $t$, one has the immediate equation
\begin{equation}\label{32}
2i\pi\,(P\,\K({\mathfrak S})\,Q+Q\,\K({\mathfrak S})\,P)=\K\l((x_2\partial_1-x_1\partial_2)\,{\mathfrak S}\r),
\end{equation}
the proof of which necessitates nothing more than an integration by parts in the second term. An equivalent equation is
\begin{equation}\label{33}
-<{\mathfrak S},\,v'\,\otimes\,(x_2 u)>+<{\mathfrak S},\,(x_1v)\,\otimes\,u'>=
<(2i\pi\Rr){\mathfrak S},\,\,v\,\otimes\,u>:
\end{equation}
the change of sign in the two terms originates from our having transposed the operators $\partial_1$
and $\partial_2$.\\

\begin{theorem}\label{theo31}
The series $\sum_{(a,c)=1} I_{a,c}$ is weakly convergent in the space $\Rr^4{\mathcal S}'(\R^2)$ of continuous linear forms on $\Rr^4{\mathcal S}(\R^2)$.
\end{theorem}

From (\ref{222}), one obtains, if $v,u\in {\mathcal S}(\R)$ and $ac\neq 0$, the equation
\begin{equation}\label{34}
<\overset{\circ}{I}_{a,c}\,v\otimes\,u>=a^{-1}\,A_{a,c}(v)\,B_{a,c}(u),
\end{equation}
with
\begin{align}
A_{a,c}(v)&=\intR \exp\l(\frac{i\pi c\, x_1^2}{a}\r)\,\exp\l(\frac{2i\pi\,x_1\sqrt{2M}}{a}\r)\,v(x_1)\,dx_1,\nonumber\\
B_{a,c}(u)&=\intR \exp\l(\frac{i\pi c\, x_2^2}{a}\r)\,u(x_2)\,dx_2.
\end{align}
As a function of $a,c$, the product $A_{a,c}(v)\,B_{a,c}(u)$ is bounded if $v$ and $u$ are summable. One can replace the condition $ac\neq 0$ by $a\neq 0$.\\

Integrations by parts yield
\begin{align}\label{36}
A_{a,c}(Pv)=-&\intR \exp\l(\frac{i\pi c\, x_1^2}{a}\r)\,\exp\l(\frac{2i\pi\,x_1\sqrt{2M}}{a}\r)\nonumber\\
&\l[\frac{cx_1}{a}+\frac{\sqrt{2M}}{a}\r]\, v(x_1)\,dx_1,\nonumber\\
B_{a,c}(Pu)=-&\intR \exp\l(\frac{i\pi c\, x_2^2}{a}\r)\,\frac{cx_2}{a}\,u(x_2)\,dx_2.
\end{align}
while the computation of $A_{a,c}(x_1v)$ or that of $B_{a,c}(x_2u)$ is immediate. If one applies these transformations inside (\ref{32}) or (\ref{33}), one observes that the terms containing $x_1x_2$ as a factor cancel off, so that
\begin{equation}\label{37}
<\Rr\,\overset{\circ}{I}_{a,c},\,v\,\otimes\,u>=\frac{c\sqrt{2M}}{a^2}\,
<\overset{\circ}{I}_{a,c})\,v\,\otimes\,(x_2u)>.
\end{equation}
Iterating the operation twice, one gets
\begin{equation}\label{38}
<\Rr^4\,\overset{\circ}{I}_{a,c},\,v\,\otimes\,u>=\l(\frac{c\sqrt{2M}}{a^2}\r)^4\,
<\overset{\circ}{I}_{a,c},\,v\,\otimes\,(x_2^4\,u)>.
\end{equation}
Not forgetting the extra coefficient $a^{-1}$ in front of the right-hand side of (\ref{34}), one sees that the series $\sum I_{a,c}$, when limited to the terms such that $|a|\geq |c|$, is weakly convergent in ${\mathcal S}'(\R^2)$. One may note that using $\Rr^2$ in place of $\Rr^4$ would just fail, while
using $\Rr^3$ would do.\\

Next, it is immediate that one has if ${\mathfrak S}\in{\mathcal S}'(\R^2)$ and $v,u\in {\mathcal S}(\R)$ the identity
\begin{equation}
<{\mathfrak S},\,\widehat{v}\,\otimes\,\widehat{u}>=<\Feuc {\mathfrak S},\,v\,\otimes\,u>
=<-i\,\Omega\l(\sm{0}{1}{-1}{0}\r)\,{\mathfrak S},\,v\,\otimes\,u>.
\end{equation}
On the other hand, from (\ref{219}),
\begin{equation}
\Omega\l(\sm{0}{1}{-1}{0}\r)\,I_{a,c}=I_{-c,a},
\end{equation}
so that
\begin{equation}\label{311}
<I_{-c,a},\,v\otimes\,u>=i\,<I_{a,c},\,\widehat{v}\,\otimes\,\widehat{u}>.
\end{equation}
Trading the pair $v,u$ for the pair $\widehat{v},\widehat{u}$, one obtains the convergence of the series $\sum I_{a,c}$ limited to the terms such that $|c|\geq|a|$. The proof of Theorem \ref{theo31} is complete.\\

Let us display explicitly the identities obtained:
\begin{align}\label{312}
<\overset{\circ}{I}_{a,c},\,v\,\otimes\,u>&=a^{-1}\,A_{a,c}(v)\,B_{a,c}(u),\nonumber\\
<\overset{\circ}{I}_{a,c},\,v\otimes\,u>&=i\,e^{\frac{2i\pi M}{ac}}\,c^{-1}\,A_{-c,a}({\mathcal F}^{-1}v)\,B_{-c,a}({\mathcal F}^{-1}u).
\end{align}
The harmless coefficient $e^{\frac{2i\pi M}{ac}}$ is the ratio of the two factors neglected when replacing $I_{a,c}$ or $I_{-c,a}$ by $\overset{\circ}{I}_{a,c}$ or $\overset{\circ}{I}_{-c,a}$. In particular,\begin{equation}\label{313}
\l|<\overset{\circ}{I}_{a,c},\,v\,\otimes\,u>\r|\leq \begin{cases}
|a|^{-1}\,\Vert v \Vert_{\overset{}{L^1}}\,\Vert u \Vert_{\overset{}{L^1}},\qquad &a\neq 0,\\
|c|^{-1}\,\Vert {\mathcal F}^{-1}v \Vert_{\overset{}{L^1}}\,\Vert {\mathcal F}^{-1}u \Vert_{\overset{}{L^1}},\qquad &c\neq 0.
\end{cases}
\end{equation}\\

Using the fact that $\Vert w_q \Vert_{\overset{}{L^1}}=q^{-1}\,\Vert w \Vert_{\overset{}{L^1}}$ for every $w\in {\mathcal S}'(\R)$, the case when $|a|\geq |c|$ and $q>1$ is immediate.

\section{The Hecke operator $T_p$ and its powers}\label{sec4}

The calculations in the present section are an exact replica of those made in \cite[Section 5]{untRam} in the Maass case. The algebra is the same. This coincidence does not extend to the analytic developments making up the rest of either paper.\\

N.B. The notation in this section is incompatible with that of \cite{untRam}: the operators denoted as $T_p^{\mathrm{dist}},\, \tau[\beta],\,\sigma_r$ here and there, the roles of which in the two domains are totally similar, are not identical.\\

The Poincar\'e series $\l(\Theta_m{\mathfrak T}_M\r)(z)$ (\ref{21}) is a modular form of weight $m+1$, actually a cusp-form. On the class of holomorphic functions $f$ invariant under the change $z\mapsto z+1$, one defines for $p$ prime \cite[(6.13)]{iwatop} the operator $T_p$ such that
\begin{equation}\label{41}
(T_pf)(z)=p^m\,f(pz)+\frac{1}{p}\,\sum_{s\mm p} f\l(\frac{z+s}{p}\r)\,.
\end{equation}
It preserves the space of cusp-forms of given weight $m+1$. Cusp-forms of Hecke type are those which are moreover joint eigenfunctions of the collection of Hecke operators. In other words, a cusp-form with the Fourier expansion $f(z)=\sum_{n\geq 1}b_n\,e^{2i\pi nz}$ is of Hecke type if, for any prime $p$, $T_pf$ is a multiple of $f$, of necessity by the factor $\frac{b_p}{b_1}$  \cite[(6.39)]{iwatop}.\\

\begin{lemma}\label{lem41}
Define on distributions ${\mathfrak S}$ invariant under $\Omega\l(N^{\bullet}\r)$ the operator $T_p^{\mathrm{plane}}$ such that
\begin{equation}
\l(T_p^{\mathrm{plane}}{\mathfrak S}\r)(x)=p^{\frac{m}{2}-1}{\mathfrak T}\l(\frac{x}{\sqrt{p}}\r)+p^{\frac{m}{2}}\sum_{s\mm p}{\mathfrak S}(x\sqrt{p})\,
\exp\l(i\pi\,\beta\,|x|^2\r)\,,
\end{equation}
in other words, denoting as $\tau[\beta]$ the operator $\Omega\l(\sm{1}{0}{\beta}{1}\r)$ of multiplication by $e^{i\pi\beta\,|x|2}$,
\begin{equation}
p^{-\frac{m}{2}}T_p^{\mathrm{plane}}=p^{-\hf-i\pi\E}+\sum_{s\mm p} \tau[s]\, p^{-\hf+i\pi\E}\,.
\end{equation}
If ${\mathfrak S}\in {\mathcal S}'(\R^2)$ is invariant under $\Omega\l(N^{\bullet}\r)$, one has for $m=3,5,\dots$ the identity
\begin{equation}\label{44}
T_p\,\l(\Theta_m\,{\mathfrak S}\r)=
\Theta_m\l(T_p^{\mathrm{plane}}{\mathfrak S}\r)\,.
\end{equation}\\
\end{lemma}

\begin{proof}
One has
\begin{multline}\label{45}
\l(T_p\Theta_m{\mathfrak S}\r)(z)\\
=\int_{\R^2}(x_1+ix_2)^m{\mathfrak S}(x)\,\l[p^m\,e^{i\pi pz\,|x|^2}+\frac{1}{p}\sum_{s\mm p}
 \exp\l(\frac{i\pi (z+s)\,|x|^2}{p}\r)\r] dx\\
 =\int_{\R^2} (x_1+ix_2)^m\,e^{i\pi z\,|x|^2}\l[p^{\frac{m}{2}-1}{\mathfrak T}\l(\frac{x}{\sqrt{p}}\r)+p^{\frac{m}{2}}\,{\mathfrak S}(x\sqrt{p})\,e^{i\pi s|x|^2}\r] dx.
\end{multline}\\
\end{proof}

The following lemma prepares for the proof of Theorem \ref{theo43}.\\

\begin{lemma}\label{lem42}
Denote as $H$ the operator $p^{-\hf-i\pi\E}$.
Recall that, for $\beta\in \R$, the operator $\tau[\beta]$ is the operator of multiplication by the function $\exp\l(i\pi\beta\,|x|^2\r)$. If one introduces for every $j\in \Z$ the space ${\mathrm{Inv}}(p^j)$ consisting of tempered distributions invariant under $\tau[p^j]$, the operator $H^{\ell}$ acts for every $\ell\in \Z$ from ${\mathrm{Inv}}(p^j)$ to ${\mathrm{Inv}}(p^{j-\ell})$. Given two linear endomorphisms $A_1$ and $A_2$ of ${\mathrm{Inv}}(1)$, write $A_1\sim A_2$ if the two operators (which may well extend to a larger space) coincide there. On the other hand, introduce the operators
\begin{equation}\label{46}
\sigma_r=\frac{1}{p^r}\sum_{s\mm p^r} \tau\l[\frac{s}{p^r}\r],\qquad \sigma_r^{(\ell)}=\frac{1}{p^r}\sum_{s\mm p^r} \tau\l[sp^{\ell-r}\r],
\end{equation}
the first (the case $\ell=0$ of $\sigma_r^{(\ell)}$) on ${\mathrm{Inv}}(1)$, the second on ${\mathrm{Inv}}\l(p^{\ell}\r)$. One has
\begin{equation}
H=p^{-\hf-i\pi\E},\quad H^{-1}=p^{\hf+i\pi\E},\qquad
{\mathrm{so\,\,that\,\,}}\qquad
p^{-\frac{m}{2}}T_p^{\mathrm{plane}}=H+\sigma_1^{(1)}\,H^{-1}.
\end{equation}
One has for every pair $r,\ell$ of non-negative integers
\begin{equation}
H^{\ell}\sigma_r\sim \sigma_{r+\ell}H^{\ell},\qquad H^{-\ell}\sigma_r\sim \sigma_r^{(\ell)}H^{-\ell},\qquad \sigma_r H^{-1}\sigma_1\sim H^{-1}\sigma_{r+1}.
\end{equation}\\
\end{lemma}

\begin{proof}
A distribution ${\mathfrak S}$ lies in ${\mathrm{Inv}}\l(p^{\ell}\r)$ if it is invariant under the multiplication by $\exp\l(i\pi p^{\ell}|x|^2\r)$: in particular, $\psi_M\in {\mathrm{Inv}}(1)$ for $M=1,2,\dots$.\\

That $H^{\ell}$ sends ${\mathrm{Inv}}(p^j)$ to ${\mathrm{Inv}}(p^{j-\ell})$ is immediate. So is the fact that $H\tau[\beta]=\tau\l[\frac{\beta}{p}\r]H$ does. Write then
\begin{equation}\label{49}
H\,\sigma_r=\frac{1}{p^r}\sum_{s\mm p^r} \tau\l[\frac{s}{p^{r+1}}\r]H,\qquad
\sigma_{r+1}H=\frac{1}{p^{r+1}}\sum_{s_1\mm p^{r+1}}\tau\l[\frac{s_1}{p^{r+1}}\r]H.
\end{equation}
The first operator makes sense on ${\mathrm{Inv}}(1)$ because if ${\mathfrak S}$ lies in this space, $H{\mathfrak S}\in {\mathrm{Inv}}\l(\frac{1}{p}\r)$, and the knowledge of $b$ mod $p^r$ implies that of $\frac{b}{p^{r+1}}$ up to a multiple of $\frac{1}{p}$. That $H\,\sigma_r$ and $\sigma_{r+1}H$ agree on ${\mathrm{Inv}}(1)$ follows. By induction on $\ell$, $H^{\ell}\sigma_r\sim \sigma_{r+\ell}H^{\ell}$. Next,
\begin{align}
\l(H^{-\ell}\sigma_r{\mathfrak S}\r)(x)&=H^{-\ell}\l[x\mapsto p^{-r}\sum_{s\mm p} {\mathfrak S}(x)\,\exp\l(\frac{i\pi\,s\,|x|^2}{p^r}\r)\r]\nonumber\\
&=p^{\ell-r}\sum_{s\mm p^r}{\mathfrak S}(p^{\frac{\ell}{2}}x)\,\exp\l(i\pi s\,p^{\ell-r}|x|^2\r)\nonumber\\
&=p^{-r}\sum_{s\mm p^r} \exp\l(i\pi s\,p^{\ell-r}|x|^2\r)\l(p^{\ell(\hf+i\pi\E)}{\mathfrak S}\r)(x),
\end{align}
so that $H^{-\ell}\sigma_r\sim \sigma_r^{(\ell)}H^{-\ell}$.\\

Finally, if ${\mathfrak S}\in {\mathrm{Inv}}(1)$, $\sigma_1{\mathfrak S}\in {\mathrm{Inv}}(p^{-1})$, so that $H^{-1}\sigma_1{\mathfrak S}\in {\mathrm{Inv}}(1)$, and the operator $\sigma_r\,H^{-1}\sigma_1$ is well-defined on this space. One has
\begin{multline}
(\sigma_r\,H^{-1}\sigma_1{\mathfrak S})(x_1,x_2)=\sigma_r\l[(x_1,x_2)\mapsto p\,(\sigma_1{\mathfrak S})\l(p^{\hf}x\r)\r]\\
=p^{-r+1}\sum_{s\mm p^r}(\sigma_1{\mathfrak S})\l(p^{\hf}x\r)
\exp\l(\frac{i\pi s\,|x|^2}{p^r}\r)
=p\,{\mathfrak T}(p^{\hf}x)=\l(H^{-1}{\mathfrak T}\r)(x)
\end{multline}
with
\begin{align}\label{411}
{\mathfrak T}(x)&=p^{-r}\sum_{s\mm p^r} (\sigma_1{\mathfrak S})(x) \exp\l(\frac{i\pi s\,|x|^2}{p^{r+1}}\r)\nonumber\\
&=p^{-r-1}\sum_{\begin{array}{c} s\mm p^r \\ s'\mm p\end{array}} {\mathfrak S}(x)\,\exp\l(\frac{i\pi s\,|x|^2}{p^{r+1}}+\frac{i\pi s'\,|x|^2}{p}\r)\,.
\end{align}
As $s$ and $s'$ run through the classes indicated as a subscript, $s+p^rs'$ describes a full class modulo $p^{r+1}$, so that the right-hand side of (\ref{411}) is the same as $(H^{-1}\sigma_{r+1}{\mathfrak S})(x_1,x_2)$. In other words,
$\sigma_r H^{-1}\sigma_1\sim H^{-1}\sigma_{r+1}$.\\
\end{proof}

\begin{theorem}\label{theo43}
Fixing a prime $p$, set $H=p^{-\hf-i\pi\E}$ and
\begin{equation}\label{413}
\widetilde{T}=p^{-\frac{m}{2}}T_p^{\mathrm{plane}}=H+\sigma_1^{(1)}H^{-1}
=H+H^{-1}\sigma_1.
\end{equation}
Given $k=1,2,\dots$ and $\ell$ such that $0\leq \ell\leq k$, there are non-negative integers  $\alpha_{k,\ell}^{(0)},\,\alpha_{k,\ell}^{(1)},\,\dots,\,\alpha_{k,\ell}^{(\ell)}$\,,
satisfying the conditions:\\

(i) $\alpha_{k,\ell}^{(0)}+\alpha_{k,\ell}^{(1)}+\dots +\alpha_{j,\ell}^{(\ell)}=\l(\begin{smallmatrix} k\\ \ell \end{smallmatrix}\r)$ for all $k,\ell$\,,\\

(ii) $2\ell -k -r\leq 0$ whenever $\alpha_{k,\ell}^{(r)}\neq 0$,\\

\noindent
such that one has the identity (between two operators on the space of $N$-invariant distributions ${\mathfrak S}$)
\begin{equation}\label{414}
\widetilde{T}^k=\sum_{\ell=0}^k H^{k-2\ell}\,\l(\alpha_{k,\ell}^{(0)}\,I+\alpha_{k,\ell}^{(1)}\,\sigma_1+\dots +\alpha_{k,\ell}^{(\ell)}\,\sigma_{\ell}\r).
\end{equation}\\
\end{theorem}

\begin{proof}
By induction. Assuming that the given formula holds, we write $\widetilde{T}^{k+1}=\widetilde{T}^k(H+H^{-1}\sigma_1)$, using the equations $\sigma_r\,H\sim H\,\sigma_{r-1}$ ($r\geq 1$) and $\sigma_r\,H^{-1}\sigma_1\sim H^{-1}\sigma_{r+1}$. We obtain
\begin{align}\label{417}
\widetilde{T}^{k+1}&=\sum_{\ell=0}^k H^{k+1-2\ell}\l(\alpha_{k,\ell}^{(0)}\,I+\alpha_{k,\ell}^{(1)}\,I
+\alpha_{k,\ell}^{(2)}\,\sigma_1+\dots +\alpha_{k,\ell}^{(\ell)}\,\sigma_{\ell-1}\r)\nonumber\\
&+\sum_{\ell=0}^k H^{k-1-2\ell}\l(\alpha_{k,\ell}^{(0)}\,\sigma_1+\alpha_{k,\ell}^{(1)}\,\sigma_2+\dots +\alpha_{k,\ell}^{(\ell)}\,\sigma_{\ell+1}\r),
\end{align}
or
\begin{align}
\widetilde{T}^{k+1}&=\sum_{\ell=0}^k H^{k+1-2\ell}\l(\alpha_{k,\ell}^{(0)}\,I+\alpha_{k,\ell}^{(1)}\,I
+\alpha_{k,\ell}^{(2)}\,\sigma_1+\dots +\alpha_{k,\ell}^{(\ell)}\,\sigma_{\ell-1}\r)\nonumber\\
&+\sum_{\ell=1}^{k+1} H^{k+1-2\ell}\l(\alpha_{k,\ell-1}^{(0)}\,\sigma_1+\alpha_{k,\ell-1}^{(1)}\,\sigma_2+\dots +\alpha_{k,\ell-1}^{(\ell-1)}\,\sigma_{\ell}\r).
\end{align}
The point (i) follows, using $\l(\begin{smallmatrix} k \\ \ell \end{smallmatrix}\r)+\l(\begin{smallmatrix} k \\ \ell-1 \end{smallmatrix}\r)=\l(\begin{smallmatrix} k+1 \\ \ell \end{smallmatrix}\r)$. Next, one observes that, in the expansion of $\widetilde{T}^{k+1}$,
the term $H^{k+1-2\ell}_r$ is the sum of two terms originating (in the process of obtaining $\widetilde{T}^{k+1}$ from $\widetilde{T}^k$) from the terms $H^{k-2\ell}\sigma_{r+1}$ and $H^{k-2\ell+2}\sigma_{r-1}$. The condition $2\ell-(k+1)-r\leq 0$ is certainly true if either $2\ell-k-(r+1)\leq 0$ or $(2\ell-2)-k-(r-1)\leq 0$, which proves the point (ii) by induction.\\
\end{proof}

\section{The main estimate}

We need to analyze the effect of powers of the Hecke operator, or of the individual terms $H^{k-2\ell}\sigma_r$ of the sum (\ref{414}), on ${\mathfrak T}_M=\sum_{(a,c)=1} I_{a,c}$.
This will be based on the equations (\ref{22}), which yield immediately, using
\begin{equation}
<\Omega\l(\sm{1}{0}{c}{1}\r){\mathfrak S},\,v\,\otimes\,\,u>=\int_{\R^2} v(x_1)\,\l(\Omega\l(\sm{1}{0}{c}{1}\r)\,{\mathfrak S}\r)(x_1,x_2)\,u(x_2)\,dx_1\,dx_2,
\end{equation}
the equations
\begin{align}\label{52}
<\Omega\l(\sm{1}{0}{c}{1}\r){\mathfrak S},\,v\otimes\,u>&=<{\mathfrak S},\,(e^{i\pi cx_1^2}\,v)\,\otimes\,(e^{i\pi cx_2^2}\,u)>,\nonumber\\
<\Omega\l(\sm{q^{-1}}{0}{0}{q}\r){\mathfrak S},\,v\,\otimes\,u>&=<{\mathfrak S},\,(qv_q)\,\otimes\,(qu_q)>,\nonumber\\
<\Omega\l(\sm{0}{1}{-1}{0}\r){\mathfrak S},\,\,v\,\otimes\,u>&=-i\,
<{\mathfrak S},\,\widehat{v}\,\otimes\,\widehat{u}>,
\end{align}
with $v_q(x_1)=v(qx_1)$.\\

This set of identities is fully similar (though even easier to establish) to the set of equations which, in pseudodifferential analysis, relate the metaplectic representation to the geometric action of $SL(2,\R)$ in ${\mathcal S}'(\R^2)$.

\begin{theorem}\label{theo51}
Given a prime $p$, and $q=p^n$ with $n\in \Z$, the distribution $q^{-1-2i\pi\E}\,\sigma_r\,{\mathfrak T}_M$ remains in a weakly bounded subset, independent of $n$ and of $r=0,1,\dots$, of the space of continuous linear forms on $\Rr^4{\mathcal S}(\R^2)$.\\
\end{theorem}

\begin{proof}
We rewrite the first two identities as the handier
\begin{align}\label{54}
<q^{-1-2i\pi \E}{\mathfrak S},\,v\,\otimes\,u>&=\int_{\R^2}v(x_1)\,q^{-2}{\mathfrak S}(\frac{x_1}{q},\frac{x_2}{q})\,u(x_2)\,dx_2\nonumber\\
&=<{\mathfrak S},\,v_q\,\otimes\,u_q)>
\end{align}
and
\begin{equation}\label{55}
<\tau[\beta]\,{\mathfrak S},\,v\otimes\,u>=
<{\mathfrak S},\,(e^{i\pi \beta x_1^2}v)\,\otimes\,(e^{i\pi \beta x_2^2}u)>.
\end{equation}\\

Replacing ${\mathfrak T}_M$ by $\Rr^4{\mathfrak T}_M$ has been shown in Theorem \ref{theo31} to take care of the $(a,c)$-summation. We address now the question of dependence on $q$ of the distribution $q^{-1-2i\pi\E}\,\sigma_r\,\Rr^4{\mathfrak T}_M$. To tackle the $q$-dependence, we may specialize in the individual terms $\overset{\circ}{I}_{a,c}$ of the series (\ref{219}), but not before we have used Lemma \ref{lem42} to change the order of the two factors $q^{-1-2i\pi\E}$ and $\sigma_r$. The operator $q^{-1-2i\pi\E}=p^{-n(1+2i\pi\E)}$ is the one denoted as $H^{2n}$ in Lemma \ref{lem42}, so that
\begin{equation}
q^{-1-2i\pi\E}\,\sigma_r=\sigma'_r\,q^{-1-2i\pi\E},
\end{equation}
where $\sigma'_r$ is the self-transpose operator defined as
\begin{equation}
\sigma'_r=\begin{cases}\sigma_{r+2n}\quad &{\mathrm{if}}\,\,n\geq 0,\,\, {\mathrm{or}} \,\,q\geq 1,\\
\sigma_r^{(-2n)}\quad &{\mathrm{if}}\,\,n\leq 0,\,\, {\mathrm{or}}\,\, q\leq 1.
\end{cases}
\end{equation}
Just remember from Lemma \ref{lem42} that $\sigma'_r$ is an arithmetic average of operators $\tau[\beta]$: $\beta$ can be chosen to lie in $[0,1[$ if $q\geq 1$, only in $[0,p^{-2n}[=
[0,q^{-2}[$ if $q<1$.\\

When this has been done, we can consider in place of $q^{-1-2i\pi\E}\,\sigma_r\,\Rr^4{\mathfrak T}_M$ the terms $\sigma'_r\,q^{-1-2i\pi\E}\,\Rr^4\overset{\circ}{I}_{a,c}$ of the series expansion
of its version modified in the way just indicated. Note that, considering the individual terms of the two series, this is not the same as $q^{-1-2i\pi\E}\,\sigma_r\,\Rr^4\overset{\circ}{I}_{a,c}$: applying the commutation formulas in Lemma \ref{lem42} demands that one should be dealing with a distribution invariant under the multiplication by $e^{i\pi\,|x|^2}$. As a final simplification, we may drop for awhile the rotation operator $\Rr^4$ which commutes with all operators present in the algebraic calculations, just remembering when needed that $u$ is divisible by $x_2^4$ in ${\mathcal S}(\R)$. We are left with the following problem: show that, given $v,u\in {\mathcal S}(\R)$, and $q=p^n$ with $n\in \Z$,
\begin{equation}\label{57}
<\tau[\beta]\,q^{-1-2i\pi\E}\,\overset{\circ}{I}_{a,c},\,\,v\,\otimes\,u>=
<\overset{\circ}{I}_{a,c},\,\l(e^{i\pi\beta x_1^2}\,v\r)_q\,\otimes\,\l(e^{i\pi\beta x_2^2}\,u\r)_q>
\end{equation}
is bounded uniformly with the respect to $\beta\geq 0$ in the following two cases: when $q\geq 1$ and $\beta\leq 1$, and when $q\leq 1$ and $\beta\leq q^{-2}$. In view of (\ref{38}), we may assume that the function $u$ is divisible by $x_2^4$.\\

The analysis of this problem will also depend on whether $|a|\geq |c|$ or $|a|<|c|$. We must not lose the $(a,c)$-summability, so we must use the first equation (\ref{312}) when $|a|\geq |c|$, the second when $|a|< |c|$. The equation (\ref{313}) will suffice in three cases (out of four), but it will be necessary to go back to (\ref{312}) in the case when $|a|>|c|$ and $q<1$.\\

When $q\geq 1$, one has $\beta\leq 1$ and a function such as $e^{i\pi\beta x_2^2}\,u$ remains, as $\beta$ varies, in a bounded subset of ${\mathcal S}(\R)$: besides, $\Vert w_q\Vert_{\overset{}{L^1}}=q^{-1}\, \Vert w\Vert_{\overset{}{L^1}}$, so we are done in the case when $q\geq 1$ and $|a|\geq |c|$. The same goes if $|a|\leq |c|$, using this time the second equation (\ref{313}) and the fact that $\Vert {\mathcal F}^{-1}w_q\Vert_{\overset{}{L^1}}=
\Vert w\Vert_{\overset{}{L^1}}$. In this last case, the condition $q\geq 1$ is not necessary.\\

The last case is the one for which $q<1$ and $|a|> |c|$: recall that we are now dealing with an arithmetic average of operators $\tau[\beta]$ with $\beta\leq q^{-2}$. We write, using (\ref{57}) and (\ref{312}),
\begin{multline}
<\tau[\beta]\,q^{-1-2i\pi\E}\overset{\circ}{I}_{a,c},\,\,v\,\otimes\,u>=
<\overset{\circ}{I}_{a,c},\,\,\l(e^{i\pi\beta x_1^2}\,v\r)_q\,\otimes\,\l(e^{i\pi\beta x_2^2}\,u\r)_q>\\
=i\,e^{\frac{2i\pi M}{ac}}\,<\overset{\circ}{I}_{-c,a},\,\,\l({\mathcal F}^{-1}(e^{i\pi\beta x_1^2}\,v)_q\r),\otimes\,\l({\mathcal F}^{-1}(e^{i\pi\beta x_2^2}\,u)_q\r)>.
\end{multline}
We use now (\ref{34}), which leads us to the question of bounding independently of $\beta\leq q^{-2}$ the expression
\begin{multline}\label{59}
B_{-c,a}\l({\mathcal F}^{-1}(e^{i\pi\beta x_2^2}\,u)_q\r)\\
=\intR \exp\l(-\frac{2i\pi a x_2^2}{c}\r) \l({\mathcal F}^{-1}(e^{i\pi\beta x_2^2}\,u)_q\r)(x_2)\,dx_2\\
=\intR \exp\l(-\frac{2i\pi a x_2^2}{c}\r)dx_2\,\intR e^{i\pi\beta q^2t^2}\,u(qt)\,e^{2i\pi t x_2}\,dt\\
=\frac{1}{q}\,\intR \exp\l(-\frac{2i\pi a x_2^2}{c}\r)dx_2\,\intR e^{i\pi\beta t^2}\,u(t)\,e^{\frac{2i\pi t x_2}{q}}\,dt\\
=\frac{1}{q}\,(\frac{a}{c})^{-\hf}\,\intR e^{i\pi \beta t^2}\,\exp\l(\frac{2i\pi ct^2}{aq^2}\r) u(t)\,dt.
\end{multline}

Recall that using the operator $\Rr^4$, so as to ensure the $(a,c)$-summability for $|c|\geq |a|$,
we benefit at the same time (\ref{38}) from the possibility to assume that $u$ is divisible by $x_2^4$:
writing $u(x_2)=x_2\,u_1(x_2)$ will suffice for our purpose. Then, with the help of the integration by parts associated to the equation
\begin{equation}
\exp\l(i\pi (\beta+\frac{2c}{aq^2})\,t^2\r)=
(2i\pi t)^{-1}\,(\beta+\frac{2c}{aq^2})^{-1}\,\frac{d}{dt}\,\l[\exp\l(i\pi (\beta+\frac{2c}{aq^2})\,t^2\r)\r],
\end{equation}
we save an extra factor $(\beta+\frac{2c}{aq^2})^{-1}\leq \frac{aq^2}{c}$, the last inequality originating from the inequality $\beta<q^{-2}$. This factor just covers the (bad) factor $\frac{1}{q}$ on the left of the right-hand side of (\ref{59}), together with the similar one associated to the $A_{a,c}$-term.\\

\end{proof}

\section{The Ramanujan-Deligne theorem}\label{sec6}


Though the series ${\mathfrak T}_M=\sum_{g\in \Gamma/N^{\bullet}} \Omega(g)\,\psi_M
=\sum_{(a,c)=1}I_{a,c}$ considered in (\ref{219}) does not converge in ${\mathcal S}'(\R^2)$, the series of its $\Theta_m$-transforms does converge for every $m$ in the space of holomorphic functions in $\H$, and its sum is given, with $\generic \in \Gamma$, as
\begin{equation}\label{61}
\l(\Theta_m{\mathfrak T}_M\r)(z)=-i\,(2M)^{\frac{m}{2}}\sum_{(a,c)=1}
(-dz+b)^{-m-1}\exp\l(2i\pi M\frac{-cz+a}{-dz+b}\r),
\end{equation}
a standard Poincar\'e series. This expansion is a consequence of the equation
\begin{equation}
\l(\Theta_m\psi_M\r)(z)=-i(2M)^{\frac{m}{2}} e^{2i\pi Mz}
\end{equation}
and of Proposition \ref{prop21}.\\

\begin{theorem}\label{theo61}
(Deligne) The coefficients $b_p$ of the Fourier series decomposition of a holomorphic cusp-form $f$ of even weight $2k=m+1$ and Hecke type, normalized by the condition $b_1=1$, satisfy for $p$ prime the estimate $|b_p|\leq 2\,p^{k-\hf}$.\\
\end{theorem}

\begin{proof}
Let $M_0$ be the dimension of the linear space of cusp-forms of weight $2k$: then \cite[p.54]{iwatop}, $f$ is of necessity a linear combination of the Poincar\'e series
$\Theta_m{\mathfrak T}_M$ (cf.\,(\ref{61})) with $M\leq M_0$. For each such $M$, and large $N$, we consider the image of ${\mathfrak T}_M$ under $\l(p^{-\frac{m}{2}}\,T_p^{\mathrm{plane}}\r)^{2N}$, i.e., using Theorem \ref{theo43}, under the operator
\begin{equation}
\l(p^{-\hf-i\pi\E}+p^{\hf+i\pi\E}\,\sigma_1\r)^{2N}=\sum_{\ell=0}^{2N} \sum_{0\leq r\leq \ell}\alpha_{2N,\ell}^{(r)}\,\,\,p^{(\ell-N)(1+2i\pi \E)}\,\sigma_r\,.
\end{equation}\\

Fix $p$, set ${\mathfrak T}_M^{[4]}=\Rr^4\,{\mathfrak T}_M$ and recall that, with $q=p^{\ell-N}$ and $p^{(\ell-N)(1+2i\pi \E)}=q^{1+2i\pi \E}$, Theorem \ref{theo51} gives for the expression $\Lan q^{1+2i\pi\E}\,\sigma_r\,{\mathfrak T}_M^{[4]},\,h\Ran$ a bound $\Zert\,h\,\Zert$, for some continuous norm $\Zert\,\,\Zert$ on ${\mathcal S}(\R^2)$ independent of $q$ and $r$, which can also be chosen independent of $M\leq M_0$. It follows then from (\ref{413}) and (\ref{414}) that
\begin{equation}
\l|\Lan \l(p^{-\hf-i\pi\E}+p^{\hf+i\pi\E}\,\sigma_1\r)^{2N}{\mathfrak T}_M^{[4]},\,h\Ran\r|\leq \Zert h \Zert\,\times\,\sum_{\ell=0}^{2N}\sum_{0\leq r\leq \ell} \alpha^{(r)}_{2N,\ell}\,.
\end{equation}
Using Theorem \ref{theo43} again, one has
\begin{equation}\label{65}
\sum_{\ell=0}^{2N}\sum_{0\leq r\leq \ell} \alpha^{(r)}_{2N,\ell}=
\sum_{\ell=0}^{2N}\begin{pmatrix} 2N \\ \ell\end{pmatrix}=2^{2N}.
\end{equation}
The product $2^{-2N}\,\l(p^{-\frac{m}{2}}T_p^{\mathrm{plane}}\r)^{2N}\,{\mathfrak T}_M^{[4]}$ thus remains for every $M$, as $N\to \infty$, in a bounded subset of the space of tempered distributions.\\

Let $f$ be a Hecke eigenform of weight $2k$. According to what has been recalled in the beginning of this proof, $f$ coincides with a (finite) linear combination $\sum_{M\leq M_0}\,\beta_M\,\Theta_m\,{\mathfrak T}_M$, and
\begin{equation}
T_p^{2N}f=\sum_M \beta_M\,T_p^{2N}\,\Theta_m\,{\mathfrak T}_M.
\end{equation}
On the other hand, $(2i\pi\Rr)(x_1+ix_2)^{-m}=im\,(x_1+ix_2)^{-m}$ and
$\Theta_m\,(\Rr^4\,{\mathfrak T}_M)=(\frac{m}{2\pi})^4\Theta_m\,{\mathfrak T}_M$.\\

If $f$ admits the expansion $f(z)=\sum_{n\geq 1} b_n\,e^{2i\pi nz}$ with $b_1=1$, one thus has, using the fundamental property of the Hecke operator recalled immediately after (\ref{41}), to wit $T_pf=b_pf$, and not forgetting the factor $p^{\frac{m}{2}}=p^{k-\hf}$ in (\ref{413}),

\begin{align}
2^{-2N}(\frac{m}{2\pi})^4\,\l(p^{-\frac{m}{2}}\,b_p\r)^{2N}f&=2^{-2N}\sum_M \beta_M\,\l(p^{-\frac{m}{2}}T_p\r)^{2N}\,\Theta_m\,{\mathfrak T}_M^{[4]}\nonumber\\
&=\sum_M \beta_M\,\Theta_m\l[2^{-2N}\l(p^{-\frac{m}{2}}T_p^{\mathrm{dist}}\r)^{2N}\,{\mathfrak T}_M^{[4]}\r].
\end{align}
From the estimate that follows immediately (\ref{65}), this is the image under $\Theta_m$ of a distribution which remains in a bounded subset of ${\mathcal S}'(\R^2)$. The coefficient
$2^{-2N}(\frac{m}{2\pi})^4\,\l(p^{-\frac{m}{2}}\,b_p\r)^{2N}$ on the left-hand side thus remains bounded as $N\to \infty$. Taking a $(2N)$th root, we have obtained the desired estimate $|b_p|\leq 2p^{\frac{m}{2}}$.\\

\end{proof}

\section{The Ramanujan conjecture for Maass forms: a brief comparison}

The proof of the Ramanujan-Petersson conjecture in the Maass case \cite{untRam} and the present proof of the Ramanujan-Deligne theorem \cite{del} follow totally similar ways: everything is similar, and everything is different. In both cases, developing the analysis in the plane, rather than the half-plane, is essential. But it is in the present paper the plane given with its Euclidean structure that must be used, while in the case of Maass forms it is the symplectic plane. The pair of operators $(\Rr,-\E)$ present here must be traded in the Maass case for the pair $(\E,\E^{\natural})$, with $2i\pi\E^{\natural}=x_1\partial_1-x_2\partial_2$. The algebraic part (the present Section 4, culminating in Proposition 4.3), is based in the Maass case on a set of commutation relations identical to those in Lemma \ref{lem42}, involving however a set of operators distinct from the set $\{H,\,\tau[\beta]\}$. The algebraic calculations are then identical.\\

In the Maass case, it is the symbols in the sense of the Weyl operator-calculus that play the role ascribed here to integral kernels: then, the bilinear object $v\,\otimes\,u$ becomes the so-called Wigner function, the hermitian object defined as
\begin{equation}
\Wig(v,\,u)(x,\xi)=\intR \overline{v}(x+t)\,u(x-t)\,e^{2i\pi t\xi}\,dt.
\end{equation}
At the end, the analysis in the Maass case requires some more developments: while the eigenvalues of the rotation operator are integers, Maass forms exist for ``mysterious'' eigenvalues of the operator $\E$. Also, in the Maass case, cusp-forms are no longer linear combinations of Poincar\'e series: instead, some spectral-theoretic developments are needed.\\

We renounced writing an axiomatic version of the Ramanujan conjecture that would cover both cases, for several reasons, one of which is that we have already submitted our proof of the Ramanujan-Petersson conjecture elsewhere. It might start with some common terminology: may we suggest the word "Schwartz symbolic operator-calculus" for the representation of operators by their integral kernels ? we might then write a mental-two column exposition. The covariance property demands in each case using a pair of representations, one on the space of symbols and the other (essentially the metaplectic representation in both cases) on the space of functions on the line the operators are tested on.
This is not completely classical for the following two reasons: first, it goes beyond geometric quantization theory (the representation $\Omega$ is not purely geometric). Next, it is essential to enrich the representations under consideration by taking their semi-direct products with the Heisenberg representation: the crucial identity (\ref{32}) and its analogue in the Weyl calculus depend on this extension. All this will subsist if one replaces the group $SL(2,\R)$ by the symplectic group ${\mathrm{Sp}}(2n,\R)$. All elements of the necessary construction show themselves naturally, just replacing the one-dimensional Weyl calculus by the $n$-dimensional one, and the upper half-plane by
the complex tube over the cone of positive-definite symmetric matrices: but we have not checked details. The basic arithmetic aspects would be covered in the book \cite{maass}. \\

\end{document}